\documentclass[12pt]{amsart}

\usepackage[english]{babel}
\usepackage{vmargin,graphicx,amsmath,amssymb,color,subfigure,enumerate,csquotes, dsfont, mathrsfs}
\numberwithin{equation}{section}

\usepackage{color}
\definecolor{gr}{rgb}   {0.,   0.69,   0.23 }
\definecolor{bl}{rgb}   {0.,   0.5,   1. }
\definecolor{mg}{rgb}   {0.85,  0.,    0.85}
\definecolor{or}{rgb}   {0.9,  0.5,   0.}

\definecolor{webred}{rgb}{0.75,0,0}
\definecolor{webgreen}{rgb}{0,0.75,0}
\usepackage[citecolor=webgreen,colorlinks=true,linkcolor=webred]{hyperref}

\newtheorem{theorem}{Theorem}[section]
\newtheorem{lemma}[theorem]{Lemma}

\newtheorem{remark}[theorem]{Remark}

\newtheorem{proposition}[theorem]{Proposition}

\newtheorem{definition}[theorem]{Definition\rm}

\newcommand{\dist}{\mathsf{dist}}

\renewcommand{\Re}{\mathsf{Re}}
\renewcommand{\Im}{\mathsf{Im}}

\newcommand{\Dom}{\mathsf{Dom}}

\newcommand{\N}{\mathbb{N}}

\newcommand{\R}{\mathbb{R}}
\newcommand{\C}{\mathbb{C}}
\newcommand{\A}{\mathbf{A}}

\newcommand{\B}{\mathbf{B}}

\newcommand{\sH}{\mathsf{H}}
\newcommand{\sL}{\mathsf{L}}
\newcommand{\sB}{\mathsf{B}}
\newcommand{\Id}{\mathsf{Id}}
\newcommand{\eps}{\varepsilon}
\newcommand{\dx}{\,\mathrm{d}}
\newcommand{\cD}{\mathcal{D}}
\newcommand{\cV}{\mathcal{V}}

\begin{document}
\title[Holomorphic extension]{Holomorphic extension of the de Gennes function}
\author{V. Bonnaillie-No\"el}
\address[V. Bonnaillie-No\"el]{D\'epartement de Math\'ematiques et Applications (DMA - UMR 8553), PSL Research University, CNRS, ENS Paris, 45 rue d'Ulm, F-75230 Paris cedex 05, France}
\email{bonnaillie@math.cnrs.fr}
\author{F. H\'erau}
\address[F. H\'erau]{LMJL - UMR6629, Universit\'e de Nantes, 2 rue de la Houssini\`ere, BP 92208, 44322 Nantes Cedex 3, France}
\email{herau@univ-nantes.fr}
\author{N. Raymond}
\address[N. Raymond]{IRMAR, Univ. Rennes 1, CNRS, Campus de Beaulieu, F-35042 Rennes cedex, France}
\email{nicolas.raymond@univ-rennes1.fr}
\date{\today}
\maketitle

\begin{abstract}
This note is devoted to prove that the de Gennes function has a holomorphic extension on a strip containing the real axis.
\end{abstract}

\maketitle

\section{Introduction}

\subsection{About the de Gennes operator}
The de Gennes operator plays an important role in the investigation of the magnetic Schr\"odinger operator. Consider the half-plane 
\[\R^2_{+}=\{(s,t)\in\R^2 : t>0\}\] 
and a magnetic field $\B(s,t)=1$. We also introduce an associated vector potential $\A(s,t)=(0,-t)$ such that $\B(s,t)=\nabla\times\A(s,t)$. Then,  we define $\mathcal{L}_{\A}$ the Neumann realization of the differential operator $(-i\nabla+\A)^2$ acting on $\sL^2(\R^2_{+})$. By using the partial Fourier transform in $s$, we obtain the direct integral:
\[\mathcal{F}\mathcal{L}_{\A}\mathcal{F}^{-1}=\int_{\R}^{\oplus}\mathfrak{L}_{\xi}\dx\xi\,,\]
where, for $\xi\in\R$, $\mathfrak{L}_{\xi}$ denotes the Neumann realization on $\R_{+}$ of the differential operator $-\partial_{t}^2+(\xi-t)^2$. The self-adjoint operator $\mathfrak{L}_{\xi}$ is called the de Gennes operator with parameter $\xi\in\R$. Let us give a precise definition of this operator. For $\xi\in\R$, we introduce the sesquilinear form
\[\forall \psi\in \sB^1(\R_{+})\,,\quad q_{\xi}(\varphi,\psi)=\int_{\R_{+}} \varphi'\overline{\psi}'+(t-\xi)^2\varphi\overline{\psi}\dx t\,,\]
where
\[\sB^1(\R_{+})=\left\{\psi\in\sH^1(\R_{+}): t\psi\in\sL^2(\R_{+})\right\}\,.\]
This form is continuous on the Hilbert space $\sB^1(\R_{+})$ and Hermitian. Up to the addition of a constant, $q_{\xi}$ is coercive on $\sB^1(\R_{+})$. Therefore, in virtue of the Lax-Milgram representation theorem (see for instance \cite[Theorem 3.4]{Hel13}), we can consider the associated closed (and self-adjoint, since the form is Hermitian) operator $\mathfrak{L}_{\xi}$ and its domain is
\[\Dom(\mathfrak{L}_{\xi})=\left\{\psi\in\sB^2(\R_{+}) : \psi'(0)=0\right\}\subset\sB^1(\R_{+})\,,\]
where
\[\sB^2(\R_{+})=\left\{\psi\in\sH^2(\R_{+}) : t^2\psi\in\sL^2(\R_{+})\right\}\,.\]
Since $\sB^1(\R_{+})$ is compactly embedded in $\sL^2(\R_{+})$, we deduce that $\mathfrak{L}_{\xi}$ has compact resolvent and we can consider the non-decreasing sequence of its eigenvalues $(\mu_{k}(\xi))_{k\in\N^*}$. Each eigenspace is of dimension one (due to the Neumann  condition for instance). 

\begin{definition}
We call the function $\R\ni\xi\mapsto\mu_{1}(\xi)\in\R$ the de Gennes function. For shortness, we let $\mu=\mu_{1}$.
\end{definition}
In the terminology of Kato's perturbation theory (see \cite[Section VII.2]{Kato66}), the family of self-adjoint operators $(\mathfrak{L}_{\xi})_{\xi\in\R}$ is analytic of type $(A)$. In other words, for all $\xi_{0}\in\R$, the family $(\mathfrak{L}_{\xi})_{\xi\in\R}$ can be extended into a family of closed operators $(\mathfrak{L}_{\xi})_{\xi\in\cV}$ with $\cV$ a complex open ball containing $\xi_{0}$ such that
\begin{enumerate}[-]
\item the domain $\Dom(\mathfrak{L}_{\xi})$ does not depend on $\xi\in\cV$,
\item for all $\psi\in\Dom(\mathfrak{L}_{\xi_{0}})$, the map $\cV\ni\xi\mapsto \mathfrak{L}_{\xi}\psi$ is holomorphic.
\end{enumerate}
By using Kato's theory and since $\mu$ is a simple eigenvalue, $\mu$ is \emph{real} analytic, or equivalently, for all given $\xi_{0}\in\R$, $\mu$ has a local holomorphic extension in a neighborhood of $\xi_{0}$. We refer to \cite[Section 2.4]{Ray17} for a direct proof. We also recall that $\mu$ has a minimum (see \cite{DauHel} or \cite[Proposition 3.2.2]{FouHel10}). We let $\displaystyle{\Theta_{0}=\min_{\xi\in\R} \mu(\xi)}$. 

This note answers the following question: \emph{Can the de Gennes function $\mu$ be holomorphically extended on a strip about the real axis?}

\subsection{Motivations and result}
The aim of this note is to prove the following theorem.
\begin{theorem}\label{theo.main}
There exist $\eps>0$ and $F$ a holomorphic function on the strip 
\[S_{\eps}:=\{\xi\in\C : |\Im\xi|<\eps\},\]
such that, for all $\xi\in\R$, $\mu(\xi)=F(\xi)$. Moreover, for all $\xi\in S_{\eps}$, we have
\[\Re F(\xi)\geq \mu(\Re \xi)-(\Im\xi)^2\,,\]
and $F(\xi)$ belongs to the discrete spectrum of $\mathfrak{L}_{\xi}$.
\end{theorem}
\begin{remark}
As we can see from the proof, for all $k\in\N^*$, $\mu_{k}$ has a holomorphic extension on a strip about the real axis. Of course, the size of this strip depends on $k$ and is expected to shrink when $k$ becomes large.
\end{remark}
\begin{remark}
The method used in this note can be applied to the family of Montgomery operators 
\[-\partial^2_{t}+\left(\xi-\frac{t^{n+1}}{n+1}\right)^2\,,\]
acting on $\sL^2(\R)$, with $n\in\N^*$ and $\xi\in\R$. Their eigenvalues have holomorphic extensions on a strip about the real axis. The Montgomery operators appear in the case of vanishing magnetic fields.
\end{remark}

This theorem is motivated by various spectral questions. Firstly, it plays an important role in the investigation of the resonances induced by local perturbations of $\mathcal{L}_{\A}$. Indeed, these resonances can often be defined by analytic dilations (see for instance \cite[Chapter 16]{HS96}) and the existence of a holomorphic extension of the band functions $\mu_{k}$ could be useful to reveal new magnetic spectral phenomena. Secondly, it is also strongly related to complex WKB analysis (as we have shown in \cite{BHR16}). In particular, the WKB constructions in \cite[Section 1.2.2]{BHR16} are {\it a priori} local and an accurate knowledge of the holomorphy strip would allow the extend the domain of validity of these constructions (see \cite[Section 4.2]{BHR16} where the size of the holomorphy strip is involved). Moreover, this holomorphic extension is crucial in the study of the semiclassical magnetic tunneling effect when there are symmetries. Actually, this effect can only be fully understood in the complexified phase space.

\section{Proof of the theorem}

\subsection{Preliminary considerations}
Let us prove the following separation lemma.
\begin{lemma}\label{lem.prelim}
For all $k\in\N^*$, there exists $c_{k}>0$ such that
\[\forall\xi\in\R\,,\quad \mu_{k+1}(\xi)-\mu_{k}(\xi)\geq c_{k}\,.\]
\end{lemma}
\begin{proof}
We recall that the functions $\mu_{k}$ are real analytic and that, for all integer $k$, we have $\mu_{k}<\mu_{k+1}$. As a consequence of the harmonic approximation (see for instance \cite[Section 7.1]{FouHel10} or \cite[Section 3.2]{Ray17}), we have
\begin{equation}\label{eq.+infty}
\forall k\in\N^*\,,\quad\lim_{\xi\to+\infty}\mu_{k}(\xi)=2k-1\,.
\end{equation}
Therefore, for all $k\in\N^*$, there exist $\Xi_{k}, d_{k}>0$ such that, for all $\xi\geq \Xi_{k}$, 
\[\mu_{k+1}(\xi)-\mu_{k}(\xi)\geq d_{k}\,.\]
For $\alpha>0$, let us consider $\mathfrak{L}_{-\alpha}$. By dilation, this operator is unitarily equivalent to the Neumann  realization on $\R_{+}$ of the differential operator
\begin{equation*}
\alpha^2(\alpha^{-4}D^2_{\tau}+(\tau+1)^2)\,,
\end{equation*}
where we denote $D_{t}=-i\partial_{t}$.
The potential $\R_{+}\ni\tau\mapsto(\tau+1)^2$ is minimal at $\tau=0$ and thus a variant of the harmonic approximation near $\tau=0$ shows that
\begin{equation}\label{eq.alpha-to+infty}
\forall k\in\N^*\,,\quad\mu_{k}(-\alpha)\underset{\alpha\to+\infty}{=}\alpha^2+\nu_{k} \alpha^{\frac{2}{3}}+o( \alpha^{\frac{2}{3}})\,,
\end{equation}
where $(\nu_{k})_{k\in\N^*}$ is the increasing sequence of the eigenvalues of the Neumann realization on $\R_{+}$ of $D^2_{\tau}+2\tau$. Let us briefly recall the main steps of the proof of \eqref{eq.alpha-to+infty}. The regime $\alpha\to+\infty$ is equivalent to the semiclassical regime $h=\alpha^{-2}\to 0$ and one knows that the eigenfunctions associated with the low lying eigenvalues are concentrated near the minimum of the potential $\tau=0$. Near this point, we can perform a Taylor expansion so that the operator asymptotically becomes  
\[\alpha^2(\alpha^{-4}D^2_{\tau}+1+2\tau)\,.\]
Then, we homogenize this operator with the rescaling $\tau=h^{\frac{2}{3}}\tilde\tau$ and the conclusion follows.

We deduce that
\begin{equation}\label{eq.-infty}
\lim_{\alpha\to+\infty}\left(\mu_{k+1}(-\alpha)-\mu_{k}(-\alpha)\right)=+\infty\,.
\end{equation}
Combining \eqref{eq.+infty}, \eqref{eq.-infty}, the simplicity of the $\mu_{k}$ and their continuity, the result follows.
\end{proof}

\subsection{The family $(\mathfrak{L}_{\xi})_{\xi\in S_{\eps}}$}
Let us fix $\eps>0$ and explain how the operator $\mathfrak{L}_{\xi}$ is defined for $\xi\in S_{\eps}$. We let, for all $\xi\in S_{\eps}$,
\[\forall\varphi,\psi\in\sB^1(\R_{+})\,,\quad q_{\xi}(\varphi,\psi)=\int_{\R_{+}} \varphi'\overline{\psi}'+(t-\xi)^2\varphi\overline{\psi}\dx t\,.\]
The sesquilinear form $q_{\xi}$ is well defined and continuous on $\sB^1(\R_{+})$. In order to apply the Lax-Milgram representation theorem, let us consider the quadratic form $\sB^1(\R_{+})\ni\psi\mapsto \Re q_{\xi}(\psi,\psi)$. We have
\[\Re q_{\xi}(\psi,\psi)=\|\psi'\|^2+\int_{\R_{+}}\Re(t-\xi)^2|\psi|^2\dx t\,,\]
and
\[\Re(t-\xi)^2=(t-\Re\xi)^2-(\Im\xi)^2\,.\]
In particular, we have
\begin{align}\label{eq.lbqxi}
\nonumber\Re q_{\xi}(\psi,\psi)\geq q_{\Re \xi}(\psi,\psi)-(\Im\xi)^2)\|\psi\|^2&\geq(\mu(\Re\xi)-(\Im\xi)^2)\|\psi\|^2\\
&\geq (\Theta_{0}-(\Im\xi)^2)\|\psi\|^2\,.
\end{align}
Thus, up to the addition of constant, $\Re q_{\xi}$ is coercive on $\sB^1(\R_{+})$. We deduce that, for all $\eps>0$ and all $\xi\in S_{\eps}$, there exist $C,c>0$ such that, for all $\psi\in \sB^1(\R_{+})$, 
\[\left|q_{\xi}(\psi,\psi)+C\|\psi\|^2\right|\geq c\|\psi\|^2_{\sB^1(\R_{+})}\,.\]
Therefore, we can apply the Lax-Milgram theorem and define a closed operator $\mathfrak{L}_{\xi}$. It satisfies
\[\forall\psi\in\sB^1(\R_{+})\,,\forall\varphi\in\Dom(\mathfrak{L}_{\xi})\,,\quad q_{\xi}(\varphi,\psi)=\langle\mathfrak{L}_{\xi}\varphi,\psi\rangle\,.\]
We can easily show that
\[\Dom(\mathfrak{L}_{\xi})=\left\{\psi\in\sB^2(\R_{+}): \psi'(0)=0\right\}\,,\]
so that $\mathfrak{L}_{\xi}$ has compact resolvent and $(\mathfrak{L}_{\xi})_{\xi\in S_{\eps}}$ is a holomorphic family of type $(A)$. If $\eps\in(0,\sqrt{\Theta_{0}})$, as a consequence of \eqref{eq.lbqxi}, we get that, for all $\xi\in S_{\eps}$, $\mathfrak{L}_{\xi}$ is bijective. Indeed, we have, for all $\psi\in\Dom(\mathfrak{L}_{\xi})$, 
\[(\Theta_{0}-\eps^2)\|\psi\|\leq\|\mathfrak{L}_{\xi}\psi\|\,.\]
From this, the operator $\mathfrak{L}_{\xi}$ is injective with closed range. Since $\mathfrak{L}^*_{\xi}=\mathfrak{L}_{\overline{\xi}}$, we deduce that $\mathfrak{L}^*_{\xi}$ is also injective and thus $\mathfrak{L}_{\xi}$ has a dense image. We get that, for all $\xi\in S_{\eps}$,
\[\|\mathfrak{L}_{\xi}^{-1}\|\leq (\Theta_{0}-(\Im\xi)^2)^{-1}\,.\]

\subsection{Resolvent and projection estimates}
Let us now prove the main result of this note. 

\subsubsection{Difference of resolvents}
For $0<r_{1}<r_{2}$, and $z_{0}\in\C$, we define the annulus
\[A_{r_{1}, r_{2}}(z_{0})=\{z\in\C : r_{1}<|z-z_{0}|<r_{2}\}\,.\]
We consider, for all $r>0$ and all $\xi\in S_{\eps}$,
\[\mathcal{A}_{r,\xi}:=A_{r,2r}(\mu(\Re\xi))\,.\]
Using Lemma \ref{lem.prelim} with $k=1$, we know that there exists $r_{0}>0$ such that
\begin{enumerate}[-]
\item for all $\xi\in S_{\eps}$, 
\[\mathcal{A}_{r_{0},\xi}\subset\rho(\mathfrak{L}_{\Re\xi})\,,\quad\dist(\mathcal{A}_{r_{0},\xi}, \mathsf{sp}(\mathfrak{L}_{\Re\xi}))\geq r_{0}>0\,,\]
where $\rho(\mathfrak{L}_{\Re\xi})$ denotes the resolvent set of $\mathfrak{L}_{\Re\xi}$,
\item the disk of center $\mu(\Re \xi)$ with radius $r_{0}$, denoted by $\cD(\mu(\Re\xi),r_{0})$, contains only one eigenvalue of $\mathfrak{L}_{\Re\xi}$. 
\end{enumerate}
The following proposition states an approximation of the resolvent of $\mathfrak{L}_{\xi}$ by the one of $\mathfrak{L}_{\Re\xi}$ when $\eps$ goes to $0$.
\begin{proposition}\label{prop.diff-res}
There exist $C, \eps_{0}>0$ such that for all $\eps\in(0,\eps_{0})$ and all $\xi\in S_{\eps}$, we have $\mathcal{A}_{r_{0},\xi}\subset\rho(\mathfrak{L}_{\xi})$. In addition, we have, for all $\xi\in S_{\eps}$ and $z\in\mathcal{A}_{r_{0},\xi}$,
\begin{equation}\label{eq.diff-res}
\left\|(\mathfrak{L}_{\xi}-z)^{-1}-(\mathfrak{L}_{\Re\xi}-z)^{-1}\right\|\leq C\eps\,.
\end{equation}
\end{proposition}
\begin{proof}
For all $z\in \mathcal{A}_{r_{0},\xi}$, we have
\[\mathfrak{L}_{\xi}-z=\mathfrak{L}_{\Re\xi}-2i\Im\xi (t-\Re\xi)-(\Im\xi)^2-z\,,\]
or, equivalently,
\[\mathfrak{L}_{\xi}-z=(\Id-2i\Im\xi (t-\Re\xi)(\mathfrak{L}_{\Re\xi}-z)^{-1}-(\Im\xi)^2(\mathfrak{L}_{\Re\xi}-z)^{-1})(\mathfrak{L}_{\Re\xi}-z)\,.\]
Thus, we have to show that 
\[\Id-2i\Im\xi (t-\Re\xi)(\mathfrak{L}_{\Re\xi}-z)^{-1}-(\Im\xi)^2(\mathfrak{L}_{\Re\xi}-z)^{-1}\]
 is bijective. By the spectral theorem, we get
\begin{equation}\label{eq.L2}
\|(\mathfrak{L}_{\Re\xi}-z)^{-1}\|\leq \frac{1}{\dist(z,\mathsf{sp}(\mathfrak{L}_{\Re\xi}))}\leq\frac{1}{r_{0}}\,.
\end{equation}
Let us also estimate $(t-\Re\xi)(\mathfrak{L}_{\Re\xi}-z)^{-1}$. Take $u\in\sL^2(\R_{+})$ and let 
\[v=(\mathfrak{L}_{\Re\xi}-z)^{-1}u\,.\]
We have
\[(D^2_{t}+(t-\Re\xi)^2)v=(\mathfrak{L}_{\Re\xi}-z)v=u\,,\]
and we must estimate $(t-\Re\xi)v$. Taking the scalar product with $v$, we get
\[\|(t-\Re\xi)v\|^2\leq\langle v,u\rangle\leq \|u\|\|v\|\leq r_{0}^{-1}\|u\|^2\,.\]
We find
\begin{equation}\label{eq.H1}
\|(t-\Re\xi)(\mathfrak{L}_{\Re\xi}-z)^{-1}\|\leq r_{0}^{-\frac{1}{2}}\,.
\end{equation}
From \eqref{eq.L2} and \eqref{eq.H1}, it follows that
\begin{multline*}
\|-2i\Im\xi (t-\Re\xi)(\mathfrak{L}_{\Re\xi}-z)^{-1}-(\Im\xi)^2(\mathfrak{L}_{\Re\xi}-z)^{-1}\|\\
\leq 2|\Im\xi|r_{0}^{-\frac{1}{2}}+|\Im\xi|^2r_{0}^{-1}\leq 2\eps r_{0}^{-\frac{1}{2}}+(\eps r_{0}^{-\frac{1}{2}})^2\,.
\end{multline*}
Thus, there exists $\eps_{0}>0$ such that for all $\eps\in(0,\eps_{0})$, we have, for all $\xi\in S_{\eps}$ and $z\in\mathcal{A}_{r_{0},\xi}$,
\[\|-2i\Im\xi (t-\Re\xi)(\mathfrak{L}_{\Re\xi}-z)^{-1}-(\Im\xi)^2(\mathfrak{L}_{\Re\xi}-z)^{-1}\|<1\,.\]
This shows that $\mathsf{Id}-2i\Im\xi (t-\Re\xi)(\mathfrak{L}_{\Re\xi}-z)^{-1}-(\Im\xi)^2(\mathfrak{L}_{\Re\xi}-z)^{-1}$ is bijective and thus $\mathfrak{L}_{\xi}-z$ is bijective. The estimate of the difference of the resolvents immediately follows.
\end{proof}

\subsubsection{Projections}
Let us now introduce
\[P_{\xi}=\frac{1}{2i\pi}\int_{\Gamma_{\xi}}(z-\mathfrak{L}_{\xi})^{-1}\dx z\,,\]
where $\Gamma_{\xi}$ is the circle of center $\mu(\Re \xi)$ and radius $\frac{3}{2}r_{0}$. It is classical to show, by using the resolvent and Cauchy formulas, that $P_{\xi}$ is the projection on the characteristic space associated with the eigenvalues of $\mathfrak{L}_{\xi}$ enlaced by $\Gamma_{\xi}$.
\begin{lemma}
There exists $\eps_{0}>0$ such that for all $\eps\in(0,\eps_{0})$ and all $\xi\in S_{\eps}$, the rank of $P_{\xi}$ is $1$ and that $P_{\xi}$ is holomorphic on $S_{\eps}$.
\end{lemma}
\begin{proof}
Let us fix $\xi_{0}\in S_{\eps}$. There exists $\delta>0$ such that, for all $\xi\in\cD(\xi_{0},\delta)$, we have $\xi\in S_{\eps}$, $\Gamma_{\xi_{0}}\subset \mathcal{A}_{r_{0},\xi}$ and 
\[P_{\xi}=\frac{1}{2i\pi}\int_{\Gamma_{\xi_{0}}}(z-\mathfrak{L}_{\xi})^{-1}\dx z\,.\]
This comes from the continuity of the family of circles $\Gamma_{\xi}$ and the holomorphy of the resolvent $z\mapsto (z-\mathfrak{L}_{\xi})^{-1}$. From this, we now see that $P_{\xi}$ is holomorphic on the disk $\cD(\xi_{0},\delta)$. Indeed, since $(\mathfrak{L}_{\xi})_{\xi\in S_{\eps}}$ is a holomorphic family of type $(A)$, $\cD(\xi_{0},\delta)\ni\xi\mapsto (z-\mathfrak{L}_{\xi})^{-1}$ is holomorphic, uniformly for $z\in\Gamma_{\xi_{0}}$. Finally, with \eqref{eq.diff-res}, there exists $\eps_{0}>0$ such that, for all $\eps\in(0,\eps_{0})$ and all $\xi\in S_{\eps}$,
\[\left\|P_{\xi}-P_{\Re \xi}\right\|\leq \frac{3r_{0}}{2}C\eps<1\,.\]
By a classical lemma about pairs of projections (see \cite[Section I.4.6]{Kato66}) and since the rank of $P_{\Re \xi}$ is $1$, we deduce that the rank of $P_{\xi}$ is $1$.
\end{proof}
\begin{remark}\label{rem.rem}
We can also choose $\eps_{0}>0$ so that, for all $\eps\in(0,\eps_{0})$, all $\xi\in S_{\eps}$ and all $\psi\in\sL^2(\R_+)$,
\[\left|\int_{\R_{+}} (P_{\xi}\psi)^2\dx t-\int_{\R_{+}}(P_{\Re \xi}\psi)^2\dx t\right|\leq \frac{1}{2}\|\psi\|^2\,.\]
\end{remark}

\subsubsection{Conclusion}
Since $P_{\xi}$ commutes with $\mathfrak{L}_{\xi}$, we can consider the restriction of $\mathfrak{L}_{\xi}$ to the range of $P_{\xi}$. This restriction is a linear map in dimension one. Therefore, for all $\xi\in S_{\eps}$, there exists a unique $F(\xi)\in\C$ such that 
\[\forall\psi\in\sL^2(\R_{+})\,,\quad \mathfrak{L}_{\xi}(P_{\xi}\psi)=F(\xi)P_{\xi}\psi\,.\]
Taking the scalar product with $\overline{P_{\xi}\psi}$, we get, for all $\psi\in\sL^2(\R_{+})$ and all $\xi\in S_{\eps}$,
\[F(\xi)\int_{\R_{+}}(P_{\xi}\psi)^2\dx t= \int_{\R_{+}}\mathfrak{L}_{\xi}(P_{\xi}\psi) P_{\xi}\psi\dx t\,.\]
Let $\xi_{0}\in S_{\eps}$. By Remark \ref{rem.rem}, there exists $\psi_{0}\in\sL^2(\R_{+})$ (a real normalized eigenfunction of $\mathfrak{L}_{\Re\xi_{0}}$ associated with $\mu(\Re\xi_{0})$) such that 
\[\int_{\R_{+}} (P_{\xi_{0}}\psi_{0})^2\dx t\neq 0\,.\] 
Thus there exists a neighborhood $\cV$ of $\xi_{0}$ such that, for all $\xi\in \cV$,
\[F(\xi)= \frac{\int_{\R_{+}}\mathfrak{L}_{\xi}(P_{\xi}\psi_{0}) P_{\xi}\psi_{0}\dx t}{\int_{\R_{+}}(P_{\xi}\psi_{0})^2\dx t}\,.\]
Thus $F$ is holomorphic near $\xi_{0}$. When $\xi\in\R$, we have clearly $F(\xi)=\mu(\xi)$. This, with \eqref{eq.lbqxi}, terminates the proof of Theorem \ref{theo.main}.

\def\cprime{$'$}


\end{document}